\documentclass[11pt]{article}

\usepackage{amsmath,amssymb,graphicx,color,amsthm}
\usepackage[linesnumbered,ruled]{algorithm2e}
\usepackage{tikz}

\usetikzlibrary{shapes.geometric, arrows}
\usepackage{algorithmic}

\usepackage[english]{babel}
\usepackage[utf8x]{inputenc}
\usepackage{subcaption,afterpage} 
\usepackage{mathtools}
\usepackage{amsfonts}
\usepackage[colorinlistoftodos]{todonotes}
\usepackage{enumerate}
\usepackage[round]{natbib}
\usepackage[T1]{fontenc}
\usepackage{hyperref}
\usepackage{natbib}
\usepackage{siunitx}
\newcommand{\excise}[1]{}

\newcommand\<{\langle}

\newcommand\RR{\mathbb{R}}

\newcommand\EE{\mathbb{E}}
\newcommand\MM{\mathcal{M}}
\newcommand\XX{\mathfrak{X}}

\renewcommand\>{\rangle}

\newtheorem{theorem}{Theorem}
\newtheorem{definition}{Definition}
\newtheorem{lemma}{Lemma}
\newtheorem{example}{Example}
\newtheorem{proposition}{Proposition}

\DeclareMathOperator\var{var}
\DeclareMathOperator\dive{div}

\newcommand{\RNum}[1]{\uppercase\expandafter{\romannumeral #1\relax}}

\begin{document}
	
	\title{\mbox{}\\[-11ex]Random Lie Brackets that Induce Torsion: A Model for Noisy Vector Fields}
	\vspace{-8ex}
	\author{\\[1ex]Didong Li$^1$ and Sayan Mukherjee$^{1,2}$ \\ 
		{ Departments of Mathematics and Statistical Science}$^{1}$\\
		{ Departments of Computer Science, and Biostatistics \& Bioinformatics}$^2$ \\ Duke University}
	\date{\vspace{-5ex}}
	
	\maketitle
	
We define and study a random Lie bracket that induces torsion in expectation. Almost all stochastic analysis on manifolds have assumed parallel transport. 
Mathematically this assumption is very reasonable. However, in many applied geometry and graphics problems parallel transport is not achieved, the ``change in coordinates"
are not exact due to noise. We formulate a stochastic model on a manifold for which parallel transport does not hold and analyze the consequences of this model with respect to classic
quantities studied in Riemannian geometry. We first define a stochastic lie bracket that induces a stochastic covariant derivative. We then study the connection implied by the stochastic covariant derivative and note that the stochastic lie bracket induces torsion. We then state the induced stochastic geodesic equations and a stochastic differential equation for parallel transport. We also derive the curvature tensors for our construction and a stochastic Laplace-Beltrami operator. We close with a discussion of the motivation and relevance of our construction.

\section{Introduction}
	
Stochastic processes on manifolds have been an object of interest to probabilists, harmonic analysts, statisticians, and machine learners. The basic idea across all these disciplines has been to define or characterize random processes on Riemannian manifolds. There are basically two approaches to construct or model random processes on manifolds: one can randomize paths on the manifold or randomize the geometry that the paths follow. There is extensive literature on randomizing paths on manifolds. The problem of studying paths on on a randomized geometry is less developed and is more aligned with the model we study in this paper. A byproduct of modeling stochastic processes on manifolds by randomizing paths is that these stochastic process models assume that the torsion tensor is zero, that is the Lie bracket (which encodes the geometry) is equal to the covariant derivative. The assumption that the torsion tensor is zero gives rise to the phenomena of parallel transport which allows one to connect the geometries of nearby points on the manifold. In this paper we explore stochastic models on manifolds where the torsion tensor is not zero. We introduce torsion by defining random vector fields or random diffeomorphisms which induces a random Lie bracket. Given the random Lie bracket we derive stochastic analogs of classical Riemannian structures including: torsion, parallel transport, geodesics, curvature, and the Laplace-Beltrami operator.


There is extensive literature in probability, harmonic analysis, and statistics on random processes on manifolds. Probabilists have studied Brownian motion on Riemannian manifolds \citep{stroock2000introduction,hsu2008brief}. Building on stochastic processes on manifolds, stochastic differential equations (SDEs) on manifolds are well understood \citep{ito1950stochastic,ito1953stochastic,elworthy1982stochastic,li1994stochastic,EmeryMeyer1989}. SDEs on manifolds have been applied to multiple fields, including non-linear filtering \citep{rugunanan2005stochastic} and signal processing \citep{manton2013primer}.  There is also extensive literature on Wiener measures and path integrals on Riemannian manifolds, including the the Feynman-Kac formula on Riemannian manifolds \citep{LLL2018}. A common approach in the study of Brownian motion on manifolds is to extend stochastic analysis on Euclidean space to manifolds by using the frame bundle to transfer Brownian motion in $\RR^d$ to manifolds via the so-called Eells-Elworthy-Malliavin \citep{EEM} construction. The random process we study diverges from the classical perspective of Brownian motion on Riemannian manifolds as we consider a stochastic process that will induce torsion and parallel transport does not hold. 

In harmonic analysis there is extensive work on diffusions on manifolds ranging from the theory of diffusions and semi-groups on manifolds \citep{grigoryanheat} to methodology for data analysis based on manifold assumptions including diffusion maps for dimension reduction \cite{coifman2006diffusion,CoifmanLafonLMNWZ2005PNAS1}, diffusions on non-orientable manifolds \cite{singer2011orientability}, vector diffusion maps \citep{singer2012vector}, and diffusion geometries of fiber bundles \citep{gao2019diffusion}. In the statistics and machine learning literature there has been extensive work under the rubric of manifold learning using random process on manifolds for unsupervised dimension reduction \cite{donoho03,isomap,lle,LapEigMaps2003}, supervised methods for dimension reduction \citep{mukherjee2010}, as well as inference based on Gaussian processes embedded in a manifold Gaussian process \citep{dunson2019diffusion}, and
and stochastic gradient descent \citep{feng2019uniform}. Almost all the theory as well as data analysis methods have assumed that the torsion tensor is zero and there is a connection
that allows for the flow of geometries between nearby points on the manifold. From the perspective of harmonic analysis and inference this paper explores a setting where coordinate 
changes between two points on the manifold are noisy.

Our motivation for introducing torsion to stochastic processes on manifolds arises from a data analysis application in geometric morphometrics. The objective of geometric morphometrics is to quantifying biological shape, shape variation, and covariation of shape with other biotic or abiotic variables or factors often with an eye to study evolutionary processes. Often these shapes are stored in large a database of 3-dimensional scans of surfaces such as bones or teeth \citep{Boyer:2016aa}. A classic mathematical tool to compare shapes uses diffeomorphism-based representations \citep{Dupuis:1998aa} and shapes are compared by examining the cost to continuously deform one shape into another \citep{3bd57925679c4109a40f2bbd5ae6160b,Boyer:2011aa}. These diffeomorphism-based approaches can be characterized as diffusions of fiber bundles \citep{gao2019diffusion},
hence a form of diffusion on manifolds. It is a fact that for real data the deformations do not map one shape exactly to another shape, there are errors in the correspondence map, see Figure \ref{errors} where the data are 3-dimensional scans of lemur teeth. This error in maps can naturally be thought of as the lack of parallel transport and was a strong motivation to provide a stochastic model on manifolds that has torsion.

 \begin{figure}[hbt]
\begin{center}
\includegraphics[height=2.3in]{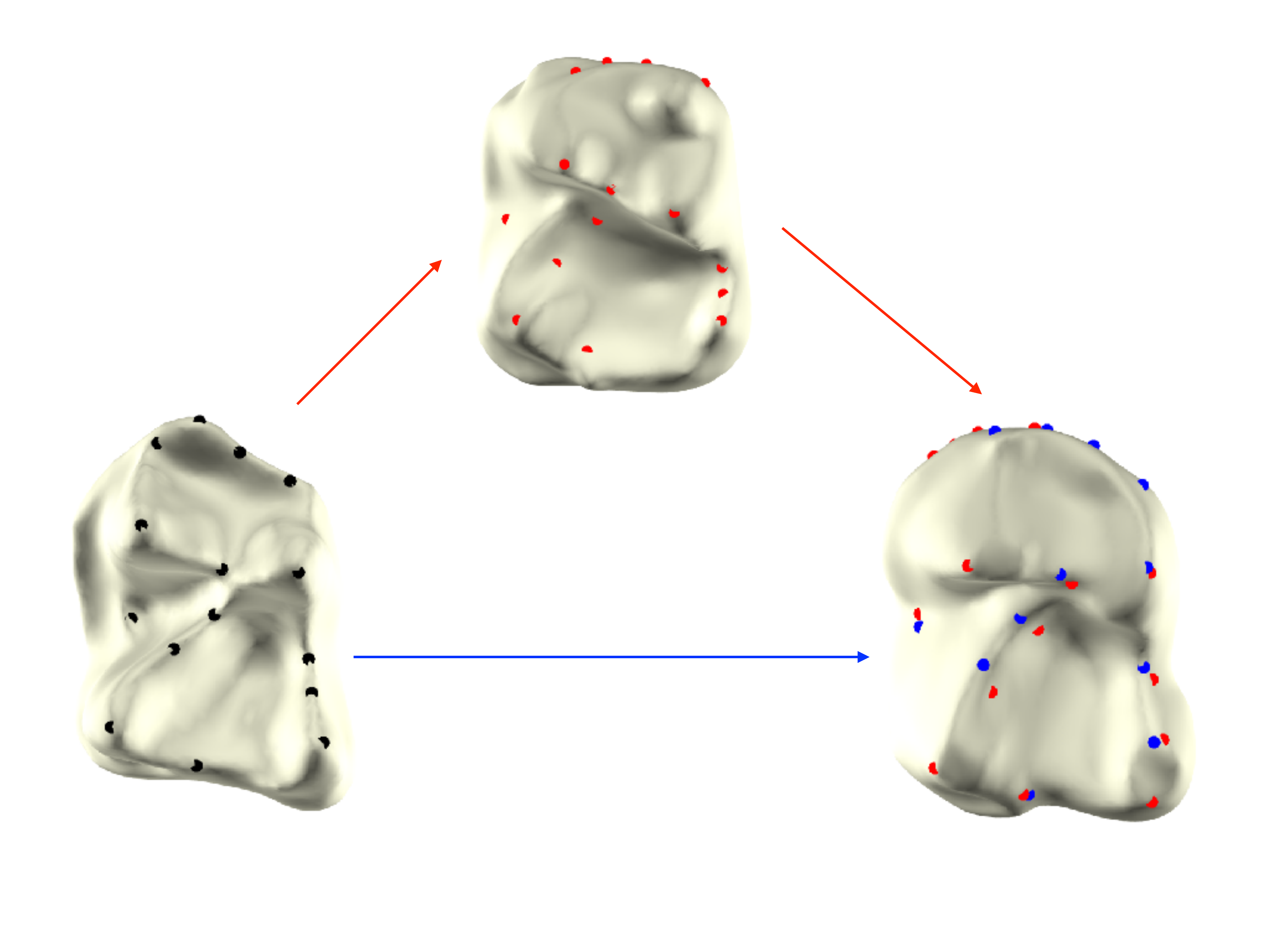}\\
\caption{\label{errors}
The black points on the far left shape are mapped to the blue points on the shape on the right. If the black points are mapped via the upper shape to the shape on the right we 
observe that they have been mapped to the red points. This example illustrates that these maps incur errors one propagates the analysis along shapes.
}
\end{center}
\end{figure}

There are two natural approaches to generalize random processes on manifolds by either randomizing sections on the manifold or making the manifold itself random. Both approaches can be thought of as particular extensions or departures form the classiic Malliavin calculus on a manifold. The basic elements required for Malliavin calculus on a manifold
are a basis manifold $V$, a fiber space $F$ on $V$, and a distribution over sections on $F$. One can generalize Malliavin calculus on a manifold by randomizing the manifold, an approach taken in \cite{khelif2013stochastic} with the introduction of $D^\infty$-stochastic manifolds. A  $D^\infty$-stochastic manifold is constructed from a family of stochastic charts
that satisfy a compatibility condition and the metric tensor, Levi-Civita connection, derivations, and curvature are derived in \cite{khelif2013stochastic}. The  $D^\infty$-stochastic manifold
is a very general construction and can have either non-zero or null torsion. The authors however focus on the setting with the analog of the classical the Levi-Civita connection and the torsion is zero. The stochastic construction in this paper is less abstract and the focus is more on introducing torsion and we make can explicitly state how the curvature tensor, parallel transport, and  Levi-Civita connection changes. In  \cite{nicolaescu2016stochastic} Gaussian ensembles on smooth sections are defined to prove a stochastic version of the Gauss-Bonnet-Chern theorem, which relates the curvature of a surface to its topology. In \cite{nicolaescu2016stochastic} it was shown that the expectation of a random current  is equal to the current defined by the Euler form. The definition of a stochastic section in  \cite{nicolaescu2016stochastic} is abstract, depending on a pullback bundle and dual bundle. Unlike our construction an explicit construction of stochastic sections in  \cite{nicolaescu2016stochastic} would be very complex. We will also show that our construction does not admit a
Gauss-Bonnet-Chern theorem so the local addition of torsion induces an obstruction to topology.


\section{Riemannian structures with random vector fields} \label{randomfields}
	
\subsection{A stochastic connection}

Consider a $m$-dimensional smooth Riemannian manifold $\MM$ equipped with a metric $g$.  Denote the space of all smooth vector fields on $\MM$ by $\XX(\MM)$, and denote $g(X,Y)$ by $\<X,Y\>$. In differential geometry a fundamental object is the connection. The connection 
formalizes the procedure of transporting data along a manifold in a consistent manner. In Riemannian geometry the canonical connection is the Levi-Civita connection which is an
affine connection typically denoted as $\nabla$ for which
\begin{eqnarray*}
\nabla g = 0,  &\quad  \ \ \mbox{the metric is preserved} \\
\nabla_X Y - \nabla_Y X = [X,Y], &\mbox{there is no torsion},
\end{eqnarray*}
in the above $X,Y$ are vector fields, $[\cdot, \cdot]$ is the Lie bracket of the fields,  and $g$ is the metric tensor. The connection encodes the geometry of the manifold and provides
a means to parallel transport tangent vectors from one point to another along a curve. The second equation above states that the connection can be stated in the form of a covariant derivative. The covariant derivative provides a calculus for taking directional derivatives of vector fields, measuring the deviation of a vector field from being parallel in a given direction.
A key idea in classical Riemannian geometry is that the geometry encoded by the Lie bracket provides the same information as the covariant derivative.

In this section we will define a stochastic connection that will share many of the nice properties of the Levi-Civita connection but will allow for the addition of torsion and state the
relation between the stochastic connection and classic deterministic one.

The key idea we will use to define a stochastic connection is a randomization of vector fields, as defined below.
\begin{definition}[A random vector field]\label{def:randvec}
		For $X\in\XX(\MM)$, a randomization of $X$, denoted by $\widetilde{X}$,  is a random vector field satisfying for any $p\in \MM$
	
	\begin{enumerate}[(a)]
		\item $\widetilde{X}(p)$ is a random vector in $T_p\MM$;
		\item $\EE[\widetilde{X}(p)]=X(p)$.
		 \end{enumerate}
	\end{definition}
There are two important properties of this randomization. First the fiber bundle structure of the field is preserved, the randomization is a random map that takes a tangent vector
to a (random) tangent vector in the same tangent space. The second condition states that this map is unbiased. In keeping with classical constructions in differential geometry, the smoothness of a vector field is necessary.  In the stochastic setting, we will require the  randomized vector field to be almost surely smooth. In this section we will provide 
examples of randomizations that satisfy the previous conditions, we will also see that there is not a unique randomization that satisfies the above properties so a rich class of 
random fields can be considered.

We will use the random vector field $\widetilde{X}$ to define a random connection. The first step of this process is to define a random differentiation  $\widetilde{D}$.
\begin{definition}[Random differentiation]\label{rand:def} 
	For any $X,Y\in\XX(\MM)$, the stochastic covariant derivative of $Y$ with respect to $X$ is defined as
	$$\widetilde D_XY\coloneqq D_{\widetilde X}\widetilde Y.$$
\end{definition}
We will use this random differentiation to generate a random connection rather than generating a random connection directly. At the end of this section we will provide an example that
directly randomizes the connection or bracket directly. 

We now check that the connection based on our random differentiation shares the same properties as an affine connection.
\begin{lemma}[Random connection] \label{lem:connection} Consider  vector fields $X, Y \in \XX(\MM)$ such that for all $f  \in C^\infty(\MM)$
	\begin{equation}\label{eqn:randfunc}
	\widetilde{fX+Y}=f\widetilde X+\widetilde Y.
	\end{equation}
Use the random differentiation $\widetilde{D}$ in Definition \ref{rand:def} to define a connection. Then for any smooth function $f\in C^\infty(\MM)$, vector fields $X,Y,Z\in \XX(\MM)$, and  scalar $a\in\RR$:
	\begin{enumerate}
		\item $\widetilde D _{fX+Y}Z = f\widetilde D_X Z+\widetilde D_YZ$;
		\item  $\widetilde D _{X}(aY+Z) =a\widetilde D_X Y+\widetilde D_XZ$;
		\item $\widetilde D_{X}fY= \widetilde X(f)\widetilde Y+f\widetilde D _{X}Y$. 
	\end{enumerate}
	
	\end{lemma}	
\noindent The first two conditions simply state the linearity of the random connection $\widetilde{D}$. The third condition is the analog of the Leibniz rule for differentiation of vector fields. 

\begin{example}
Let $\varepsilon$ be a random smooth function on $\MM$ and set $\widetilde{X}=\varepsilon X$. Two conditions stated in Definition \ref{def:randvec} hold:
$\widetilde{X}(p) \in T_p \MM$ and $\EE[\widetilde{X}(p)] = \EE[\varepsilon(p) {X}(p)] =X(p)$. It also holds that for vector fields $X, Y \in \XX(\MM)$  for all $f  \in C^\infty(\MM)$
$$\widetilde{fX+Y}=f\widetilde X+\widetilde Y.$$
We could further require that $\EE[\widetilde D_XY]=D_XY,\:\forall X,Y\in\mathfrak{X}(M)$, that is the connection itself behave like the Levi-Civita connection in expectation which
implies that
\begin{equation*}
\EE[\varepsilon^2(p)]=1,\:\forall p\in \MM.
\end{equation*} 
Requiring both	 $\EE[\widetilde{X}(p)]$ and  $\EE[\varepsilon^2(p)]=1$ results in $\var(\varepsilon(p))=0$, so there is no randomization. So the  random connection cannot in expectation
give back the Levi-Civita connection.
\end{example}
A natural question is assuming  $D$ is the Levi-Civita connection with respect to Riemannian metric $g$ what is the relation between $D$ and $\widetilde{D}$ and what is the relation
between the lie bracket corresponding to $D$ and the covariant derivative corresponding to $\widetilde{D}$.  This analysis will show that the random connection will induce
torsion. We start by computing the covariant derivative $\widetilde D_XY-\widetilde D_YX$ in terms of a random smooth function $\varepsilon$, our notion of noise:
\begin{align*}
		\widetilde D_XY-\widetilde D_YX&=D_{\varepsilon X}\varepsilon Y-D_{\varepsilon Y}\varepsilon X=\varepsilon\left(D_X\varepsilon Y-D_Y\varepsilon X\right)\\
		& = \varepsilon\left[\varepsilon D_X Y+X(\varepsilon)Y-\varepsilon D_YX-Y(\varepsilon)X\right]\\
		&=\varepsilon^2(D_XY-D_YX)+\varepsilon X(\varepsilon)Y-\varepsilon Y(\varepsilon)X\\
		& = \varepsilon^2[X,Y]+\varepsilon X(\varepsilon)Y-\varepsilon Y(\varepsilon)X\\
		&\neq [X,Y].
		\end{align*}
\noindent So the random covariant derivative induces torsion with respect to the Lie bracket as defined by the standard affine connection and
$$[\widetilde{X},\widetilde{Y}]=[\varepsilon X,\varepsilon Y]=\varepsilon^2 [X,Y]+\varepsilon X(\varepsilon)Y-\varepsilon Y(\varepsilon)X.$$ 
The above equation allows us to define a stochastic notion of torsion.
		\begin{definition}[Stochastic torsion]\label{def:torsion} 
			The stochastic torsion with respect to $\widetilde{D}$ is
					$$\widetilde{T}(X,Y)\coloneqq \widetilde D_XY-\widetilde D_YX-[\widetilde X,\widetilde Y].$$
From the above calculatation we conclude that $\widetilde T =\varepsilon^2 T$. In the following statements we assume the noise $\varepsilon$ has non-zero variance.
Given our construction of  $\widetilde{D}$ the following torsion terms are zero 
\begin{eqnarray*}
0 &=&  \widetilde D_XY-\widetilde D_YX-[\widetilde X,\widetilde Y]	\\	
0 &= & D_X Y- D_YX-[X, Y],
\end{eqnarray*}
while the following torsion term will not be zero
\begin{equation*}
0 \neq \widetilde D_X Y-\widetilde D_YX- [ X, Y]=(\varepsilon^2-1)[X,Y]+\varepsilon X(\varepsilon)Y-\varepsilon Y(\varepsilon)X.
\end{equation*}

		\end{definition}

We would like our connection $\widetilde{D}$ to share most of the properties of the Levi-Civita connection, the exception being that the stochastic covariant derivative induced
by our stochastic connection has torsion with respect to the standard deterministic Lie bracket.

Another property of the Levi-Civita connection is compatibility with the Riemannian metric: 
$$\<D_XY,{Z}\>+\<{Y}, D_XZ\>={X}\left(\<{Y},{Z}\>\right).$$

A natural question to ask is whether
$$\<\widetilde{D}_XY,Z\>+\<Y,\widetilde{D}_XZ\>  =  X(\<Y,Z\>),$$
which is not likely to be true because the right hand side is deterministic while the left hand side is random:
\begin{align*}
\<\widetilde{D}_XY,Z\>+\<Y,\widetilde{D}_XZ\>& = \<\varepsilon^2D_XY+\varepsilon X(\varepsilon)Y,Z\>+\<Y,\varepsilon^2D_XZ+\varepsilon X(\varepsilon)Z\>\\
& = \varepsilon^2\left(\<D_XY,Z\>+\<Y,D_XZ\>\right)+2\varepsilon X(\varepsilon)\<Y,Z\>\\
& = \varepsilon^2X(\<Y,Z\>)+2\varepsilon X(\varepsilon)\<Y,Z\>\\
&\neq X(\<Y,Z\>).
\end{align*}
Observe that 		
$X\left(\<\widetilde{Y},\widetilde{Z}\>\right)=X\left(\varepsilon^2\<Y,Z\>\right)=\varepsilon^2X\left(\<Y,Z\>\right)+2\varepsilon X(\varepsilon)\<Y,Z\>$, so we conclude that 
$$\<\widetilde{D}_XY,Z\>+\<Y,\widetilde{D}_XZ\>=X\left(\<\widetilde{Y},\widetilde{Z}\>\right) \neq  X(\<Y,Z\>),$$		
and the randomization does not recover the classical quantity $X(\<Y,Z\>)$.

The following theorem summarizes the properties that our stochastic connection satisfies that have natural analogs in the  Levi-Civita case, the exception is in our construction all
these quantities are random.
		\begin{theorem}\label{St:Levi-Civita}
			Assume $\widetilde X=\varepsilon X$ where $\varepsilon$ is a almost surely randm smooth function on $\MM$,  define $\widetilde{D}_XY := D_{\widetilde{X}}\widetilde{Y}$, then for any $a\in\RR$, $f\in C^\infty(\MM)$ and $X,Y,Z\in\mathfrak{X}(\MM)$, we have
			\begin{enumerate}
				\item $\widetilde D_{fX+Y}Z=f\widetilde D_XZ+\widetilde D_YZ$.
				\item $\widetilde D_X(aY+Z)=a\widetilde D_XY+\widetilde D_XZ$.
				\item $\widetilde D_X(fY)=f\widetilde D_XY+\widetilde X(f)\widetilde Y$.
				\item $\widetilde D_XY-\widetilde D_YX-[\widetilde{X},\widetilde{Y}]=0$.
				\item $\<\widetilde D_XY,\widetilde{Z}\>+\<\widetilde{Y},\widetilde D_XZ\>=\widetilde{X}\left(\<\widetilde{Y},\widetilde{Z}\>\right)$.
			\end{enumerate}
		\end{theorem}
\begin{proof}		
We need to prove the first three equations hold as we have all ready shown that equations (4) and (5) hold. Showing that equations (1)-(3) hold results from
applying equation \eqref{eqn:randfunc} to each of the first three equations above:
\begin{enumerate}
\item 	$
\widetilde D _{fX+Y}Z = D_{\widetilde{fX+Y}}\widetilde Z=D_{f\widetilde{X}+\widetilde{Y}}\widetilde{Z}=fD_{\widetilde X} \widetilde Z+D_{\widetilde Y}\widetilde Z=f\widetilde D_X Z+\widetilde D_YZ $.
\item $\widetilde D _{X}(aY+Z) = D_{\widetilde X}\widetilde {aY+Z}=D_{\widetilde X}(a\widetilde{Y}+\widetilde{Z})= aD_{\widetilde X} \widetilde Y+D_{\widetilde X}\widetilde Z=a\widetilde D_X Y+\widetilde D_XZ$.
\item $\widetilde D_{X}fY=D_{\widetilde X}\widetilde{fY}=D_{\widetilde X}f\widetilde Y=\widetilde X(f)\widetilde Y+fD_{\widetilde X}\widetilde Y= \widetilde X(f)\widetilde Y+f\widetilde D _{X}Y$.
	\end{enumerate}	
	\end{proof}	
		
Equations (1)-(3) imply that $\widetilde {D}$ is a connection, (4) states that $\widetilde {D}$ is stochastic torsion free while (5) shows that $\widetilde {D}$ is compatible with the stochastic Riemannian metric $\widetilde g(X,Y) = \<\widetilde X,\widetilde Y\>$, formally represented by $\widetilde{D}g=0$ .

To help fix the idea of a random vector field we provide a natural construction of the random function or field $\varepsilon$. This is obviously not the only construction and it would be interesting to further understand in greater detail how the geometric quantities we consider in this paper vary with the specification of the stochastic vector field. The following example
is possibly the most obvious construction.

\begin{example}[A random vector field]\label{eg:epsilon}

Assume $(\MM,g)$ is a compact Riemannian manifold and $\Delta$ is the Laplace-Beltrami operator with eigenvalues $0=\lambda_0\leq\lambda_1\leq\cdots$. We will use 
as bases the orthonormal eigenfunctions $\{\psi_i\}_{i=1}^\infty$ of $L^2(\MM,dV_g)$, that is
$$\Delta \psi_i = \lambda \psi_i,\;\<\psi_i,\psi_j\>_{L^2}=\delta_{ij},\;\forall i,j.$$
Let $X_1,X_2,\cdots$ be i.i.d Gaussian random variables $X_i\sim N(0,\sigma_k^2)$. The random functions we consider are defined as 
$$\varepsilon(x)= \sum_{i=1}^\infty X_i\psi_i(x)+1, \quad \forall x\in \MM.$$
If $\lim_{i\rightarrow \infty} i^\alpha\sigma_i^2=0$, then $f\in C^\alpha(\MM)$ almost surely. In particular if $\sigma_i^2 = \frac{1}{i^{\alpha}}$ then $\varepsilon$ is almost surely $C^2$ for any $\alpha>2$. 

For this construction it holds that $\EE\left[\varepsilon(x)\right]=1$, again a property that seems natural.

There are very general constructions of  Gaussian random functions on a manifold via Gaussian measures on spaces of distributions \citep{nicolaescu2016stochastic,gelfandgeneralized, bogachev1998gaussian,hsu1997stochastic}. Our goal in this paper is not generality. We would like construction for which we can interpret and compute the random geometric objects we consider, such as geodesic equations. 		
\end{example}




\subsection{Geodesics and parallel transport}

We will derive the geodesic equations and parallel transport induced by the stochastic connection $\widetilde{D}$. First we will define the Christoffel symbol with respect to $\widetilde{D}$. We will then compute the geodesic and parallel transport equations in expectation, both of these are deterministic differential equations. We will then consider the  stochastic geodesic and parallel transport equations both of which will be stochastic differential equations. 

Let $\gamma:(-T,T)\rightarrow \MM$ be a smooth curve and denote $\gamma(0)=p$. For $p\in U$ there is a local chart $[x^1,\cdots,x^m]$ and $\gamma(t)=[\gamma^1(t),\cdots,\gamma^m(t)]$. Denote $\widetilde{\Gamma}$ as the Christoffel symbol with respect to $\widetilde{D}$. From the definition of $\widetilde{D}$
		$$\widetilde{D}_{\frac{\partial}{\partial x^i}}{\frac{\partial}{\partial x^j}}=\widetilde{\Gamma}_{ij}^k{\frac{\partial}{\partial x^k}}.$$
		
The relation between the stochastic and standard Christoffel symbols is
$$\widetilde{\Gamma}_{ij}^k=\varepsilon^2\Gamma_{ij}^k+\varepsilon\frac{\partial \varepsilon}{\partial x^i}\delta_{jk},$$
as the following calculations confirm
		\begin{align*}
		\widetilde{D}_{\frac{\partial}{\partial x^i}}{\frac{\partial}{\partial x^j}}& = D_{\varepsilon\frac{\partial}{\partial x^i}}{\varepsilon\frac{\partial}{\partial x^j}}\\
		& = \varepsilon^2D_{\frac{\partial}{\partial x^i}}{\frac{\partial}{\partial x^j}}+\varepsilon\frac{\partial \varepsilon}{\partial x^i}\frac{\partial}{\partial x^j}\\
		& = \varepsilon^2\Gamma_{ij}^k\frac{\partial}{\partial x^k}+\varepsilon\frac{\partial \varepsilon}{\partial x^i}\frac{\partial}{\partial x^j}\\
		& = \widetilde{\Gamma}^k_{ij}\frac{\partial}{\partial x^k},
		\end{align*}
		
	where $\widetilde{\Gamma}^k_{ij} = \varepsilon^2 \Gamma_{ij}^k+\varepsilon\frac{\partial \varepsilon}{\partial x^i}\delta_{jk}$. Recall that for standard Christoffel symbol, 
$$\frac{\partial g_{ij}}{\partial x^k}=g_{lj}\Gamma^{l}_{ik}+g_{il}\Gamma^{l}_{jk}.$$
In the stochastic case, a similar formula holds:

$$\frac{\partial \varepsilon^2g_{ij}}{\partial x^k}=g_{lj}\widetilde \Gamma^{l}_{ik}+g_{il}\widetilde\Gamma^{l}_{jk},$$
where the term $\varepsilon^2g_{ij}$ can be interpreted as the coefficient of  $\widetilde {g}(X,Y)=\<\widetilde{X},\widetilde{Y}\>$.

We now derive the geodesic equations. A curve	  $\gamma$ is a geodesic with respect to $\widetilde{D}$ if and only if $0 = \widetilde{D}_{\gamma'}\gamma'|_{\gamma(t)}.$ The 
 geodesic equations with respect to $\widetilde{D}$ are given by the following equation which is formally the same as the standard geodesic equation
		\begin{equation}\label{eqn:geodesic1}
		\varepsilon(t)^2\frac{d^2x^k(t)}{dt^2}+\frac{dx^i(t)}{dt}\frac{dx^j(t)}{dt}\widetilde{\Gamma}_{ij}^k(t)=0,\;,k =1,\cdots,m,
		\end{equation}
which is validated by the following calculations
		\begin{align*}
		0& = \widetilde{D}_{\gamma'}\gamma'|_{\gamma(t)} \\
		& =  \varepsilon^2D_{\gamma'}\gamma'+\varepsilon\gamma'(\varepsilon)\gamma'\\
		& = \varepsilon(t)^2\left[\frac{d^2x^k(t)}{dt^2}+\frac{dx^i(t)}{dt}\frac{dx^j(t)}{dt}\Gamma^k_{ij}(\gamma(t))\right]\frac{\partial}{\partial x^k}\bigg|_{\gamma(t)}+\varepsilon(t) \frac{d\varepsilon(t)}{dt}\frac{dx^k(t)}{dt}\frac{\partial}{\partial x^k}\bigg|_{\gamma(t)} \\
		& = \left[\varepsilon(t)^2\frac{d^2x^k(t)}{dt^2}+\varepsilon(t)^2\frac{dx^i(t)}{dt}\frac{dx^j(t)}{dt}\Gamma^k_{ij}(\gamma(t))+\varepsilon(t) \frac{d\varepsilon(t)}{dt}\frac{dx^k(t)}{dt}\right]\frac{\partial}{\partial x^k}\bigg|_{\gamma(t)} \\
		& = \left[\varepsilon(t)^2\frac{d^2x^k(t)}{dt^2}+\frac{dx^i(t)}{dt}\frac{dx^j(t)}{dt}\left(\widetilde{\Gamma}_{ij}^k-\varepsilon(t)\frac{\partial \varepsilon(t)}{\partial x^i}\delta_{jk}\right)+\varepsilon(t) \frac{d\varepsilon(t)}{dt}\frac{dx^k(t)}{dt}\right]\frac{\partial}{\partial x^k}\bigg|_{\gamma(t)} \\
		& = \left[\varepsilon(t)^2\frac{d^2x^k(t)}{dt^2}+\frac{dx^i(t)}{dt}\frac{dx^j(t)}{dt}\widetilde{\Gamma}_{ij}^k-\varepsilon(t)\frac{dx^i(t)}{dt}\frac{\partial \varepsilon(t)}{\partial x^i}\frac{dx^k(t)}{dt}+\varepsilon(t) \frac{d\varepsilon(t)}{dt}\frac{dx^k(t)}{dt}\right]\frac{\partial}{\partial x^k}\bigg|_{\gamma(t)} \\
		& = \left[\varepsilon(t)^2\frac{d^2x^k(t)}{dt^2}+\frac{dx^i(t)}{dt}\frac{dx^j(t)}{dt}\widetilde{\Gamma}_{ij}^k-\varepsilon(t)\frac{d\varepsilon(t)}{dt}\frac{dx^k(t)}{dt}+\varepsilon(t) \frac{d\varepsilon(t)}{dt}\frac{dx^k(t)}{dt}\right]\frac{\partial}{\partial x^k}\bigg|_{\gamma(t)} \\
		& = \left[\varepsilon(t)^2\frac{d^2x^k(t)}{dt^2}+\frac{dx^i(t)}{dt}\frac{dx^j(t)}{dt}\widetilde{\Gamma}_{ij}^k\right]\frac{\partial}{\partial x^k}\bigg|_{\gamma(t)}
		\end{align*}
		
		We can rewrite equation \eqref{eqn:geodesic1} in terms of $\Gamma$:
		\begin{equation}\label{eqn:geodesic2}
		\varepsilon(t)^2\frac{d^2x^k(t)}{dt^2}+\varepsilon(t)^2\frac{dx^i(t)}{dt}\frac{dx^j(t)}{dt}\Gamma^k_{ij}(\gamma(t))+\varepsilon(t) \frac{d\varepsilon(t)}{dt}\frac{dx^k(t)}{dt}=0,\;,k =1,\cdots,m,
		\end{equation}
		
Typically a geodesic is a deterministic object, the shortest path between two points on a manifold. Equations  \eqref{eqn:geodesic1} and  \eqref{eqn:geodesic2} suggest a stochastic notion of a geodesic. Before we examine random geodesics we first consider the expectation of equation \eqref{eqn:geodesic2} as the deterministic analog resulting from our stochastic connection to the classic geodesic induced by the Levi-Civita connection.

We first define a new deterministic notion of a geodesic based on the expectation of the random field used to construct  $\widetilde{D}$. 
\begin{definition}[Geodesic in expectation]
			A curve $\gamma$ is called a geodesic in expectation with respect to the stochastic connection $\widetilde{D}$ if for all $k =1,\ldots,m.$
\begin{eqnarray*}
\EE\left[ \varepsilon(t)^2\frac{d^2x^k(t)}{dt^2}+\varepsilon(t)^2\frac{dx^i(t)}{dt}\frac{dx^j(t)}{dt}\Gamma^k_{ij}(\gamma(t))+\varepsilon(t) \frac{d\varepsilon(t)}{dt}\frac{dx^k(t)}{dt}\right] &=&0, \\
\alpha(t)\left\{\frac{d^2x^k(t)}{dt^2}+\frac{dx^i(t)}{dt}\frac{dx^j(t)}{dt}\Gamma^k_{ij}(\gamma(t))\right\}+\beta(t)\frac{dx^k(t)}{dt}&=&0,
\end{eqnarray*}
where  $\EE\left[ \varepsilon(t)^2\right] = \alpha(t)$ and $\EE\left[\varepsilon(t) \frac{d\varepsilon(t)}{dt}\right] =\beta(t)$. If  $\alpha$ and $\beta$ are $C^2(\MM)$, there exists a unique solution of the above differential equations given initial conditions.			
	\end{definition}

In the noiseless case where $\varepsilon \equiv 1$,  $\alpha(t)=1$ and $\beta(t)=0$ the standard geodesic equation is recovered. If $\alpha$ and $\beta$ are both $C^2(\MM)$ or smooth, there exists a unique geodesic locally. Ideally we want an almost surely smooth (at least $C^2(\MM)$) random function $\varepsilon:\MM\rightarrow \RR$ so Gaussian processes are excluded as they do not have the required smoothness properties.

For the random vector field defined in Example \ref{eg:epsilon} we can explicitly calculate  $\alpha$ and $\beta$ 
		$$\alpha(t)=\sum_{i=1}^\infty {\psi_i(t)}^2\sigma_i^2 +1,\;\beta(t)=\sum_{i=1}^\infty \psi_i(t)\psi'_i(t)\sigma_i^2.$$
		Since the bases $\{\psi_i\}$ are all smooth, both $\alpha$ and $\beta$ are smooth, so the local existence and uniqueness of a geodesic is guaranteed.

The intuition behind parallel transport of a connection is a way of locally moving the local geometry of one point on a manifold to a nearby point, in short one should consider
parallel transport as the local realization of a connection. Given the stochastic connection $\widetilde D$ we can state two constructions of parallel transport. The first definition
is a deterministic object and is the expected parallel transport. The second is a stochastic construction of parallel transport. Before stating the two definitions we write the 
the stochastic differential equations for the parallel transport that arises from  $\widetilde D$ by setting  $0=\widetilde{D}_{\gamma'}X$
	\begin{equation}
          \label{eqn:sde}
            \begin{aligned}
 X(0) &=X_0 \\
        \varepsilon(t)^2\frac{dX^k(t)}{dt}\frac{\partial}{\partial x^k}+\varepsilon(t)^2\frac{dx^i(t)}{dt}X^j(t)\Gamma^k_{ij}\frac{\partial}{\partial x^k}+\varepsilon(t)\frac{d\varepsilon(t)}{dt}X^k(t)\frac{\partial}{\partial x^k}&=0.
    \end{aligned}
	\end{equation}

\begin{definition}[Parallel transport in expectation]
The parallel transport from a point $X = \gamma(s)$ to $X' = \gamma(t)$ is a map $\widetilde{P}_s^t: T_{\gamma(t)}\MM\rightarrow T_{\gamma(s)}\MM$, where $\widetilde{P}_s^t(v)=X(s)$ and  $X$ is the unique solution of the following linear differential equations. For all  $k=1,\ldots,n$
			$$    \begin{aligned}
			\EE\left[\varepsilon(t)^2\right]\frac{dX^k(t)}{dt}+\EE\left[\varepsilon(t)^2\right]\frac{dx^i(t)}{dt}X^j(t)\Gamma^k_{ij}+\EE\left[\varepsilon(t)\frac{d\varepsilon(t)}{dt}\right]X^k(t)&=0 \\
			X(t)&=v.
			\end{aligned}$$
\end{definition}		
From the definition we can tell that 	$\widetilde{P}_s^t$ is a linear isomorphism between tangent spaces $T_{\gamma(t)} $and $T_{\gamma(s)}$. Moreover, such parallel transport can recover the random covariant derivative in the expectation sense:
\begin{proposition}
	Let $\gamma$ be any smooth curve on $M$, for any smooth vector field $X\in\mathcal{M}$,
	$$\EE\left[\widetilde{\nabla}_{\gamma'(t)}X(t)\right]=\EE[\varepsilon(t)^2]\lim_{\Delta t\to 0}\frac{\widetilde{P}^{t+\Delta t}_t(X(t+\Delta t))-X(t)}{\Delta t}.$$
\end{proposition}
\begin{proof}
	Let $\{e_i\}$ be a basis of $T_{\gamma(0)M}$ and $e_i(t)\coloneqq \widetilde{P}^{t}_0(e_i)$. Then  $\EE\left[\widetilde{\nabla}_{\gamma'(t)}e_i(t)\right]\equiv 0$. Since $\widetilde P$ is an isomorphism, $\{e_i(t)\}$ is a basis of $T_{\gamma(t)M}$, so we can represent $X(t)$ by $X(t) = X^i(t)e_i(t)$ where $X^i(t)$ are smooth functions with respect to $t$. By Theorem \ref{St:Levi-Civita}, 
	$$\widetilde \nabla _{\gamma'(t)}X(t) = \widetilde{\gamma'(t)}(X^i(t))\widetilde e_i(t)+X^i(t)\widetilde \nabla_{\gamma'(t)}e_i(t)=\varepsilon(t)^2\frac{dX^i(t)}{dt}e_i(t)+X^i(t)\widetilde \nabla_{\gamma'(t)}e_i(t).$$
	So $$\EE\left[\widetilde \nabla _{\gamma'(t)}X(t)\right] = \EE\left[\varepsilon(t)^2\right]\frac{dX^i(t)}{dt}e_i(t)+X^i(t)\EE\left[\widetilde \nabla_{\gamma'(t)}e_i(t)\right]=\EE\left[\varepsilon(t)^2\right]\frac{dX^i(t)}{dt}e_i(t).$$
	Recall that $\widetilde P^{t+\Delta t}_t$ is an isomorphism, so
	$$\widetilde P^{t+\Delta t}_t(X(t+\Delta t))=\widetilde P^{t+\Delta t}_t(X^i(t+\Delta t)e_i(t+\Delta t))=  X^i(t+\Delta t)\widetilde P^{t+\Delta t}_te_i(t+\Delta t))=  X^i(t+\Delta t)e_i(t).$$
	As a result,
	$$\lim_{\Delta t\rightarrow 0}\frac{\widetilde{P}^{t+\Delta t}_t(X(t+\Delta t))-X(t)}{\Delta t}= \lim_{\Delta t\rightarrow 0}\frac{(X^i(t+\Delta t)-X^i(t))e_i(t)}{\Delta t}= \frac{dX^i(t)}{dt}e_i(t).$$
	So we conclude that 
	$$\EE\left[\widetilde{\nabla}_{\gamma'(t)}X(t)\right]=\EE[\varepsilon(t)^2]\lim_{\Delta t\to 0}\frac{\widetilde{P}^{t+\Delta t}_t(X(t+\Delta t))-X(t)}{\Delta t}.$$
\end{proof}
	
This result is not surprising: one cannot recover $\widetilde{\nabla}$, a random operator by the deterministic parallel transport $\widetilde P$. Instead we can recover the expectation. 

We now remove the expectation and consider a stochastic version. 
\begin{definition}[Stochastic parallel transport] The stochastic parallel transport is the solution to the follow stochastic differential equations with $k=1,\ldots, n$,
$$\varepsilon(t)^2dX^k_t+\varepsilon(t)^2\frac{dx^i(t)}{dt}\Gamma^k_{ij}(t)X^j_tdt+\varepsilon(t)\frac{d\varepsilon(t)}{dt}X^k_t=0,$$
which can be written in a more familiar notation as
\begin{equation}\label{eqn:stdsde}
dX_t=\mu(X_t,t)dt+\sigma(X_t,t) \,  d\widetilde \varepsilon(t),
\end{equation}
where the drift term is $\mu^k(X_t,t) = -\frac{dx^i(t)}{dt}\Gamma^k_{ij}(t)X^j_t$ and the diffusion term is $\sigma^k(X_t,t) =\sigma^k(X_t)= -X^k_t$ with $\widetilde{\varepsilon}=\log\varepsilon$.
\end{definition}		

Note that the equations stated in \eqref{eqn:stdsde} are not stochastic differential equations in the sense of  It\^{o} but are differential equations with random coefficients. Unlike
the It\^{o} setting $\varepsilon$ is differentiable, at least $C^2$, so the first and second order variations are finite. In other words, the Riemann-Stieltjes integral works in this situation, instead of the It\^{o} integral. Since these are linear differential equations with almost surely $C^2$ coefficients, there exists unique solution locally given initial value $X_0$, which provides a random flow starting from $X_0\in T_{\gamma(t)}M$. 

If we weaken the assumption for the smoothness of the noise $\widetilde\varepsilon$, for example assume $\widetilde\varepsilon$ is Brownian motion, $\widetilde\varepsilon(t)=B_t$, then equation \eqref{eqn:stdsde} becomes a collection of linear stochastic differential equations that can be interpreted as 
$$X_t=\int_0^t \mu(X_s,s)ds+\int_0^t \sigma(X_s,s)dB_s=\int_0^t \mu(X_s,s)ds+\lim_{n\rightarrow \infty} \sum_{[t_{i-1},t_{i}]\in \pi_n} \sigma(X_{t_{i-1}})(B_{t_i}-B_{t_{i-1}}),$$
where $\pi_n$ is a sequence of partitions of $[0,t]$ with the mesh going to zero.  
As a result, there exists a unique solution locally for a given initial condition $X_0$:
$$X_t = \Psi_t X_0,$$
where $\Psi_t$ is the fundamental matrix satisfying $\Psi_0=\mathrm{Id}$ and the homogeneous matrix SDE
$$d\Psi_t = A_t\Psi_tdt+B_td W_t,$$
where $(A_t)_{j,k}=-\frac{dx^i(t)}{dt}\Gamma^k_{ij}(t)$ and $B_t=-\mathrm{Id}$ \citep{platen2010numerical}. The intuition of $\widetilde{\varepsilon}$ being Brownian motion is that $\varepsilon$ is centered at $1$ so its log should be centered at $0$. In this Brownian motion setting, the theory of stochastic calculus applies but geometrically the resulting vector field $X_t$ given by the stochastic parallel transport is no longer a smooth vector field almost surely. This is not surprising as Brownian motion is nowhere smooth. It is the case that 
when $\varepsilon$ is the Brownian motion, which is not differentiable, stochastic parallel transport is still well defined. 

The smoothness condition on $\varepsilon$ can be weakened in this situation for two reasons. First, the stochastic parallel transport involves only first order differentiation  (see \eqref{eqn:stdsde}), so even if the noise is not differentiable one can use It\^{o} calculus. Second, the parallel transport is essentially a covariant derivative with respect to a deterministic curve $\gamma$, a one-dimensional submanifold parametrized by $t$, whose tangent vector field can be randomized by an one-dimensional Brownian motion, greatly simplifying the problem. 

However, the geodesic equation is not well defined as the derivative of the noise $\varepsilon$ is involved. A possible solution is through discretization and numerical approximation, converting equation \eqref{eqn:geodesic1} to a second order difference equations, which depends on both the choice of local coordinate chart and the discretization. Since we are focusing on globally defined geometries including stochastic covariant derivative \ref{rand:def}, stochastic torsion \ref{def:torsion} and stochastic curvature \ref{def:curvature}, certain smoothness ($C^2$) is necessary. As a result, in the remaining sections, we still assume that $\varepsilon$ is $C^2$ so that geometries can be pushed to the stochastic setting.


\subsection{Curvature}
Once we define parallel transport and geodesics the next obvious object of interest is the curvature tensor and the sectional, Ricci, and scalar curvature.

Recall that for a connection $D$ the curvature tensor $R$ for vector fields $X,Y,Z,W \in \XX(\MM)$ is a map  $R: \XX(\MM) \times  \XX(\MM) \rightarrow  \XX(\MM)$ where
 $$R(X,Y)Z = D_X D_Y Z - D_Y D_X Z - D_{[X,Y] }X.$$
 The Riemannian curvature tensor $\mathcal{R}$ is
 $${\mathcal{R}}(X,Y)Z =  D_X D_Y Z -   D_Y D_X Z.$$

\begin{definition}[Stochastic curvature tensor] 
	\label{def:curvature}
	The stochastic curvature tensor $\widetilde{R}$ with respect to $\widetilde D$ for $X,Y,Z,W\in\XX(\MM)$
		is 			
		$$\widetilde{R}(X,Y)Z:=R(\widetilde X,\widetilde Y)\widetilde Z=D_{\widetilde X}D_{\widetilde Y}\widetilde Z-D_{\widetilde Y}D_{\widetilde X}\widetilde Z-D_{{[\widetilde X,\widetilde Y]}}\widetilde Z,$$
		where $D$ is the standard affine connection and $R$ is the curvature tensor induced by $D$.
	\end{definition}
	
The stochastic curvature tensor can also be stated in terms of the classic curvature tensor as stated in the following Lemma.
\begin{lemma}
For vector fields $X,Y,Z \in \XX(\MM)$  the following relation between the stochastic and deterministic curvature tensors hold
$$\widetilde{R}(X,Y)Z=\varepsilon^3R(X,Y)Z.$$
\end{lemma}
\begin{proof}
Note that  $\widetilde R$ has three terms. We first simplify the first term and note that the  second term follows directly from the first term.
\begin{align*}
D_{\widetilde X}D_{\widetilde Y}\widetilde Z & = D_{\varepsilon X}D_{\varepsilon Y}\varepsilon Z,\\
& = \varepsilon D_{X}\left\{\varepsilon^2D_YZ+\varepsilon Y(\varepsilon)Z\right\},\\
& = \varepsilon\left\{\varepsilon^2D_XD_YZ+2\varepsilon X(\varepsilon)D_YZ+\varepsilon Y(\varepsilon) D_X Z+X(\varepsilon)Y(\varepsilon)Z+\varepsilon X(Y(\varepsilon))Z\right\},\\
& = \varepsilon^3D_XD_YZ+2\varepsilon^2 X(\varepsilon)D_YZ+\varepsilon^2 Y(\varepsilon) D_X Z+\varepsilon X(\varepsilon)Y(\varepsilon)Z+\varepsilon^2 X(Y(\varepsilon))Z.
\end{align*}
We now simplify the third term
\begin{align*}
D_{[\widetilde X,\widetilde Y]}\widetilde Z & = D_{[\varepsilon X,\varepsilon Y]}\varepsilon Z= D_{\varepsilon^2 [X,Y]+\varepsilon X(\varepsilon)Y-\varepsilon Y(\varepsilon)X}\varepsilon Z,\\
& = \varepsilon^3 D_{[X,Y]}Z+\varepsilon^2[X,Y](\varepsilon)Z+\varepsilon^2 X(\varepsilon)D_YZ+\varepsilon X(\varepsilon)Y(\varepsilon)Z-\varepsilon^2Y(\varepsilon)D_XZ-\varepsilon Y(\varepsilon)X(\varepsilon)Z,\\
& = \varepsilon^3 D_{[X,Y]}Z+\varepsilon^2X(Y(\varepsilon))Z-\varepsilon^2Y(X(\varepsilon))Z+\varepsilon^2 X(\varepsilon)D_YZ-\varepsilon^2Y(\varepsilon)D_XZ
\end{align*}
Now we can combine the above equations and cancel most of the terms
\begin{align*}
\widetilde{R}(X,Y)Z & = D_{\widetilde X}D_{\widetilde Y}\widetilde Z-D_{\widetilde Y}D_{\widetilde X}\widetilde Z-D_{{[\widetilde X,\widetilde Y]}}\widetilde Z,\\
& = \varepsilon^3D_XD_YZ+2\varepsilon^2 X(\varepsilon)D_YZ+\varepsilon^2 Y(\varepsilon) D_X Z+\varepsilon X(\varepsilon)Y(\varepsilon)Z+\varepsilon^2 X(Y(\varepsilon))Z\\
&\quad -\left\{\varepsilon^3D_YD_XZ+\varepsilon^2 X(\varepsilon) D_Y Z+2\varepsilon^2 Y(\varepsilon)D_XZ+\varepsilon Y(\varepsilon)X(\varepsilon)Z+\varepsilon^2 Y(X(\varepsilon))Z\right\},\\
&\quad - \left\{\varepsilon^3 D_{[X,Y]}Z+\varepsilon^2 X(\varepsilon)D_YZ-\varepsilon^2Y(\varepsilon)D_XZ+\varepsilon^2X(Y(\varepsilon))Z-\varepsilon^2Y(X(\varepsilon))Z\right\}\\
& = \varepsilon^3R(X,Y)Z.
\end{align*}
\end{proof}
	
The above equation can also be interpreted as the tensor property of $R$. Recall that vector fields are randomized only in the radial direction: $\widetilde{X}=\varepsilon X$, by multiplying a random function, then tensor property implies that $R$ is $C^\infty(M)$-linear.
	
	\begin{definition}[Stochastic Riemannian  curvature tensor]  The stochastic Riemannian curvature tensor for $X,Y,Z,W\in\XX(\MM)$ is
					$$\widetilde{\mathcal{R}}(X,Y,Z,W)\coloneqq \<\widetilde{R}(Z,W)X,\widetilde{Y}\>=\<R(\widetilde Z,\widetilde W)\widetilde X,\widetilde Y\>.$$
		\end{definition}

In the deterministic setting key properties of the (Riemannian) curvature tensor $\mathcal{R}(X,Y,Z,W)$ with  $X,Y,Z,W\in\XX(\MM)$ are
\begin{enumerate}
				\item Skew symmetry: $\mathcal{R}(X,Y,Z,W)=-{\mathcal{R}}(Y,X,Z,W)=-{\mathcal{R}}(X,Y,W,Z)$;
				\item Exchange symmetry: ${\mathcal{R}}(X,Y,Z,W)={\mathcal{R}}(Z,W,X,Y)$;
				\item First Bianchi identity:  $\mathcal{R}(X,Y,Z,W)+{\mathcal{R}}(Z,Y,W,X)+{\mathcal{R}}(W,Y,X,Z)=0$;
				\item Second Bianchi identity: $$D_X \mathcal{R}(Y,Z)+D_Y \mathcal{R}(Z,X)+D_Z \mathcal{R}(X,Y)=X \mathcal{R}(Y,Z)+Y\mathcal{R}(Z,X)+Z \mathcal{R}(X,Y).$$
	\end{enumerate}				
				
The following theorem states the analogous properties for the  stochastic (Riemannian) curvature tensor. We will see all the above standard properties hold
for the stochastic curvature tensor except the second Bianchi identity.
\begin{theorem}
		For any $X,Y,Z,W\in\XX(\MM)$,
			\begin{enumerate}
				\item Skew symmetry: $\widetilde{\mathcal{R}}(X,Y,Z,W)=-\widetilde{\mathcal{R}}(Y,X,Z,W)=-\widetilde{\mathcal{R}}(X,Y,W,Z)$;
				\item Exchange symmetry:  $\widetilde{\mathcal{R}}(X,Y,Z,W)=\widetilde{\mathcal{R}}(Z,W,X,Y)$;
				\item First Bianchi identity:  $\widetilde{\mathcal{R}}(X,Y,Z,W)+\widetilde{\mathcal{R}}(Z,Y,W,X)+\widetilde{\mathcal{R}}(W,Y,X,Z)=0$.
				\item Second Bianchi identity: $$D_{\widetilde{X}}\widetilde{\mathcal{R}}(Y,Z)+D_{\widetilde Y}\widetilde{\mathcal{R}}(Z,X)+D_{\widetilde{Z}}\widetilde{\mathcal{R}}(X,Y) = 3\varepsilon^3\left(X(\varepsilon)\widetilde{\mathcal{R}}(Y,Z)+Y(\varepsilon) \mathcal{\widetilde{R}}(Z,X)+Z(\varepsilon) \widetilde{\mathcal{R}}(X,Y)\right).$$
			\end{enumerate}
\end{theorem}

It is useful to write the curvature tensors with coordinates			
					$$\widetilde{R}\left(\frac{\partial}{\partial x^i},\frac{\partial}{\partial x^j}\right)\frac{\partial}{\partial x^k}=\widetilde{R}^l_{kij}\frac{\partial}{\partial x^l},\;\widetilde{R}\left(\frac{\partial}{\partial x^i},\frac{\partial}{\partial x^j},\frac{\partial}{\partial x^k}, \quad \frac{\partial}{\partial x^l}\right)=\widetilde{R}_{ijkl}.$$
and
			$$\widetilde{R}_{kij}^l=\varepsilon^3 R^l_{kij}, \quad \widetilde{R}_{ijkl}=\varepsilon^4 R_{ijkl}.$$
The covariant derivative of $\widetilde R$ denoting $\{e_i\}$ as the basis of the tangent space is
			\begin{align*}
			\widetilde{R}^{l}_{kij,h}e_l
			& = D_{\widetilde{e_h}}\left(R(\widetilde{e_i},\widetilde{e_j})\widetilde{e_k}\right)=D_{\varepsilon e_h}\left(\varepsilon^3 R(e_i,e_j)e_k\right)\\
			& = \varepsilon^4 \, D_{e_h} R(e_i,e_j)e_k+3\varepsilon^3 e_h(\varepsilon) \, R(e_i,e_j)e_k\\
			& = \varepsilon^4 R^{l}_{kij,h}e_l+3\varepsilon^3 e_h(\varepsilon) R^{l}_{kij}e_l.
			\end{align*}
So the stochastic curvature tensor in coordinates is a composition of elements from the standard curvature tensor 			
$$\widetilde{R}^{l}_{kij,h}=\varepsilon^4 R^{l}_{kij,h}+3\varepsilon^3 e_h(\varepsilon) R^{l}_{kij}.$$

There are thee notions of curvatures in differential geometry that are used to summarize the curvature tensor. The first is the Ricci curvature tensor which measures how much the
volume of a geodesic ball changes as it moves along the manifold and the classic Ricci curvature tensor is 
$$\mbox{Ric}_{ij} = \mathcal{R}^k_{ikj},$$
where $\mathcal R$ is the Riemannian curvature tensor. The sectional curvature $K$ is the curvature of two-dimensional sections of $\MM$ and can also be written in terms of the curvature tensor
$$K(X,Y) = \frac{R(X,Y,X,Y)}{\<X,X\> \<Y,Y\>-\<X,Y\>^2},$$
where $X,Y \in \XX(\MM)$. The simplest notion of curvature is the scalar curvature which is the amount by which the volume of a small geodesic ball in a Riemannian manifold deviates
from a standard ball in Euclidean space. The scalar curvature can be stated in terms of the Ricci curvature
$$S = \mbox{trace}_g \mbox{Ric}.$$

We can define the same curvature summaries for our stochastic curvature tensor. The Ricci curvature and scalar curvature differ but the sectional curvature does not.  This is because the sectional curvature $K(X,Y)$ depends on the 2-dimensional subspace spanned by $X,Y$ only, as the randomization simply rescales $X,Y$ the subspace does not change, hence the
sectional curvature is the same.
			\begin{theorem}\label{thm:seccurv}
				The stochastic sectional curvature $\widetilde K$, Ricci curvature $\widetilde{\emph{Ric}}$ and scalar curvature $\widetilde{S}$ are given by:
				\begin{enumerate}
					\item	$\widetilde{K} = K$;
					\item $\widetilde{\emph{Ric}}=\varepsilon^3 {\emph{Ric}}$;
					\item $\widetilde{S}=\varepsilon^3S$.
					\end{enumerate}		
			\end{theorem}
\begin{proof}
It suffices to calculate the sectional curvature as the other two statements are straightforward.
\begin{align*}
\widetilde{K}(X,Y)&=\frac{R(\widetilde{X},\widetilde{Y},\widetilde{X},\widetilde{Y})}{\<\widetilde{X},\widetilde{X}\>\<\widetilde{Y},\widetilde{Y}\>-\<\widetilde{X},\widetilde{Y}\>^2}\\
&=\frac{\varepsilon^4R(X,Y,X,Y)}{\varepsilon^4\<X,X\>\<Y,Y\>-\varepsilon^4\<X,Y\>^2}\\
& = \frac{R(X,Y,X,Y)}{\<X,X\>\<Y,Y\>-\<X,Y\>^2}\\
& =K(X,Y).
\end{align*}
\end{proof}			
	
Recall that the curvature form $\Omega^j_i$ is defined by $R(X,Y)=\Omega^j_i(X,Y)w^i\otimes e_j$ where $\{e_j\}$ is an orthonormal basis of tangent space and $\{w^i\}$ is its dual basis. Similarly we can defined the stochastic curvature form $\widetilde{\Omega}^j_i$ by $\widetilde R(X,Y)=\widetilde{\Omega}^j_i(X,Y)w^i\otimes e_j$. By above calculation we know that $\widetilde\Omega^j_i=\varepsilon^2\Omega^j_i$. The Chern-Gauss-Bonnet theorem connects topology and geometry of a $2p$ dimensional orientable Riemannian manifold by the following formula:
\begin{equation}
\label{GaussBonnet}
\chi(M) = \int_M \Omega,
\end{equation}
where $\chi$ is the Euler characteristic and $\Omega = \frac{(-1)^p}{2^{2p}\pi^pp!}\sum_{i_1,\cdots,i_{2p}}\delta_{1\cdots 2p}^{i_1\cdots i_{2p}}\Omega_{i_1i_2}\wedge \cdots \wedge \Omega_{i_{2p-1}i_{2p}}$.

The main goal of defining stochastic sections in  \cite{nicolaescu2016stochastic} was to provide a stochastic version of the Chern-Gauss-Bonnet theorem. We will see that for the stochastic process we consider a  stochastic version of the Chern-Gauss-Bonnet theorem will not hold. In we replace $\Omega$ with  $\widetilde \Omega$ 
$$\widetilde \Omega\coloneqq \frac{(-1)^p}{2^{2p}\pi^pp!}\sum_{i_1,\cdots,i_{2p}}\delta_{1\cdots 2p}^{i_1\cdots i_{2p}}\widetilde \Omega_{i_1i_2}\wedge \cdots \wedge \widetilde \Omega_{i_{2p-1}i_{2p}}=\varepsilon^{2p}\Omega,$$
in the Chern-Gauss-Bonnet theorem, equation \eqref{GaussBonnet}, the equality will no longer hold as the Euler characteristic is deterministic and  $\widetilde \Omega$  is random. One
can address the randomness by taking the expectation and considering the equation
$$\int_M\EE\left[ \widetilde{\Omega}\right]=\int_M \EE\left[\varepsilon^{2p}\right]\Omega\neq \chi(M)$$
so in general a Chern-Gauss-Bonnet theorem does not hold for out construction. The deviation from the Euler characteristic comes from the noise $\varepsilon$, or its high order moment, this result is a significant difference from the work in \cite{nicolaescu2016stochastic} on  stochastic Chern-Gauss-Bonnet thoerems.

\subsection{The Laplace-Beltrami Operator}
One of the most important quantities used to study time varying processes on a Riemannian manifold is the  Laplace-Beltrami operator. Here we define a stochastic Laplace-Beltrami operator. Recall that the definition of the standard Laplace-Beltrami operator is $\Delta = \dive\cdot\nabla$ where $\dive$ is the divergence and $\nabla$ is the gradient.

The obvious definition of the stochastic  Laplace-Beltrami operator is 
$$\widetilde{\Delta} = \widetilde{\dive}\cdot \widetilde{\nabla}$$
where $\widetilde{\dive}$ is the stochastic divergence and $\widetilde{\nabla}$ is the stochastic gradient. The next definition states the stochastic version of the gradient,
divergence, and Laplace-Beltrami operator.

\begin{definition}[Stochastic Laplace-Beltrami operator] For $f\in C^\infty(\MM)$ and  $X\in\XX(\MM)$.
\begin{enumerate}
\item[(1)] The stochastic gradient $\widetilde{\nabla}$ satisfies $\<X,\widetilde{\nabla}f\>= \widetilde{X}(f)$;
\item[(2)] The stochastic divergence  can be stated as the trace of the gradient 
$$\widetilde{\dive} X = \emph{trace}(\widetilde{\nabla});$$
\item[(3)] The stochastic  Laplace-Beltrami operator is $~\widetilde{\Delta} = \widetilde{\dive}\cdot \widetilde{\nabla}.$
\end{enumerate}
\end{definition}



The following lemma states the relation between the standard gradient, divergence, and Laplace-Beltrami operator and their stochastic analogs.
\begin{lemma}\label{lem:laplace}
	For any $f\in C^\infty(\MM)$ and $X\in\XX(\MM)$
	\begin{enumerate}
		\item Gradient: $\widetilde{\nabla}f = \varepsilon \nabla f$;
		\item Divergence: $\widetilde{\dive} X = \dive(\widetilde{X})$;
		\item Laplace-Beltrami: $\widetilde{\Delta} f =\varepsilon^2\Delta f+2\varepsilon\nabla f(\varepsilon)$.
	\end{enumerate}
\end{lemma}
\begin{proof} We first prove the case of the stochastic gradient and divergence.

The stochastic gradient directly follows from its definition which in local coordinates is 
	$$\widetilde{\nabla}f=\varepsilon\nabla f =\varepsilon \partial_i f g^{ij}\frac{\partial}{\partial x^j}.$$

To show the relation for the divergence we write out the following relation in local coodinates
$$D\widetilde X=\left\{\varepsilon\left(\frac{\partial X^j}{\partial x^i}+X^k\Gamma^j_{ki}\right)+\frac{\partial \varepsilon}{\partial x^i}X^j\right\}dx^i\otimes\frac{\partial}{\partial x^j}.$$
By the above computation the stochastic divergence is
$$\widetilde{\dive} X = \varepsilon\left(\frac{\partial X^i}{\partial x^i}+X^k\Gamma^i_{ki}\right)+\frac{\partial \varepsilon}{\partial x^i}X^i=\varepsilon \dive X+X(\varepsilon)=\dive(\widetilde{X}).$$

Combining the results for the gradient and divergence we get
\begin{align*}
\widetilde \Delta f &= \widetilde{\dive}\left(\widetilde{\nabla}(f)\right)=\widetilde{\dive}\left(\varepsilon \nabla f\right)=\dive(\varepsilon^2 \nabla f) \\
& = \varepsilon^2\dive(\nabla f)+\nabla f(\varepsilon^2) = \varepsilon^2\Delta f +2\varepsilon \nabla f(\varepsilon).
\end{align*}
\end{proof}	
		
Lastly in the stochastic setting an analog of the classical divergence theorem holds.
	
	\begin{theorem}\label{thm:div}
		Let $M$ be a compact, orientable Riemannian manifold with boundary and $\bf{n}$ be the inward unit normal vector of $\partial M$, then for any $X\in \XX(M)$,
		$$\int_M \widetilde\dive X \mathrm{d}V_M=-\int_{\partial M}\<{\bf n},\widetilde X\> \mathrm{d}V_M.$$
	\end{theorem}
\begin{proof}	
The theorem results from the second equation in Lemma \ref{lem:laplace} and the standard divergence theorem
$$\int_M \widetilde\dive X \mathrm{d}V_M=\int_M \dive\widetilde X \mathrm{d}V_M=-\int_{\partial M}\<{\bf n},\widetilde X\> \mathrm{d}V_M.$$
	\end{proof}

Both gradients and the Laplace-Beltrami operator have been used extensively in data science in applications such as dimensionality reduction and learning representations of
data. We are interested in applying the novel Laplace-Beltrami operator we propose to applications where the classic  Laplace-Beltrami operator has seen success.
	
\section{Discussion}

This paper introduces a novel stochastic process on manifolds where the random paths on the manifold are generated via random differentiation and random connections. The motivation for the stochastic process we propose is we are interested in the setting where information along the manifold cannot locally or globally be transported with exact fidelity,
one does not have parallel transport. The desire to model this error in coordinates comes from applications in graphics and geometric morphometrics where one cannot exactly morph
one shape into another, there are always some errors in the coordinate map. Starting with our definition of random connections  we introduce the standard analogs of quantities of interest in  Riemannian geometry: parallel transport, geodesics, curvature, and the Laplcae-Beltrami operator. 

We consider this paper as a first step in developing a stochastic calculus on manifolds where one randomizes the geometry of the manifold itself. Classic Malliavin calculus on manifolds focuses on randomizing paths rather than randomizing the geometry of the manifold. In this paper we take e very differential geometric perspective and avoid embeddings and have a coordinate-free perspective. One can take a more differential equations perspective and model the paths as parabolic stochastic differential equations embedded in an ambient space
and concentrated on a manifold. The connections between the embedded SDE and the construction we propose is of interest. There is growing literature on controlled 
rough paths \citep{lyons} and regularity structures \cite{hairer} where the random paths resemble differentiable functions unlike the random paths in classic It\^o calculus. It is of interest
to relate the smooth random differentiations and connections we have developed to the random paths perspective that has been developed in rough paths. Lastly, in this paper we focused on Riemannian geometries and affine connections, it is of interest to generalize to Finsler geometries and Ehresmann connections.
	
\subsection*{Acknowledgments} 
Sayan Mukherjee would like to thank Ingrid Daubechies, Shan Shan, Tingran Gao, Robert Adler, Jonathan Mattingly, Jiangfeng Lu, Yuliy Baryshnikov, and Juergen Jost for discussions.
Sayan Mukherjee would like to acknowledge funding from NSF DEB-1840223, NIH R01 DK116187-01, HFSP RGP0051/2017, NSF DMS 17-13012, and NSF CCF-1934964.

	\bibliographystyle{apalike}
	\bibliography{ref}

\end{document}